\documentclass[11pt,a4paper]{amsart}
\usepackage{amssymb,amsmath,amsthm}
\usepackage{inputenc}
\usepackage{graphicx}
\usepackage{subcaption}

\newtheorem{theorem}{Theorem}[]

\newtheorem{lemma}{Lemma}

\newtheorem{proposition}{Proposition}

\renewcommand{\geq}{\geqslant}
\renewcommand{\leq}{\leqslant}

\newcommand{\Tcal}{\mathcal{T}}

\newcommand{\Gcal}{\mathcal{G}}

\newcommand{\wk}{\widetilde{\kappa}}

\newcommand{\R}{\mathbb{R}}

\newcommand{\eps}{\varepsilon}

\renewcommand{\O}{\Omega}

\begin{document}

\title[Density estimates of 1-avoiding sets]{Density estimates of 1-avoiding sets via higher order correlations}


\author{Gergely Ambrus}
\address{Gergely Ambrus, Alfr\'ed R\'enyi Institute of Mathematics,
 POB 127 H-1364 Budapest, Hungary.}
\email{ambrus@renyi.hu}

\author{M\'at\'e Matolcsi}
\address{M\'at\'e Matolcsi: Budapest University of Technology and Economics (BME),
H-1111, Egry J. u. 1, Budapest, Hungary, and Alfr\'ed R\'enyi Institute of Mathematics, POB 127 H-1364 Budapest, Hungary.}
\email{matomate@renyi.hu}

\begin{abstract}
We improve the best known upper bound on the density of a planar measurable set $A$ containing no two points at unit distance to $0.25442$. We use a combination of Fourier analytic and linear programming methods to obtain the result. The estimate is achieved by means of obtaining new linear constraints on the autocorrelation function of $A$ utilizing triple-order correlations in $A$, a concept that has not been previously studied.
\end{abstract}


\thanks{G. Ambrus was supported by the NKFIH grant no. PD125502 and the Bolyai Research Fellowship of the Hungarian Academy of Sciences. M. Matolcsi was supported by the NKFIH grant no. K132097 and K129335.}



\maketitle

\section{Introduction}\label{introd}

What is the maximal upper density of a measurable planar set $A$ with no two points at distance 1? This 40-year-old question has attracted some attention recently, with a sequence of progressively improving estimates , the strongest of which currently being that of Bellitto,  P\^{e}cher, and S\'{e}dillot \cite{BPS18}, who gave the upper estimate $0.25646$. In the present article, we provide the new upper bound of $0.25442$, getting enticingly close to the upper estimate of  $0.25$ conjectured by Erd\H os. Our argument builds on the Fourier analytic method of \cite{KMOR15}. The main new ingredient is to estimate certain triple-order correlations in $A$ which lead to new linear constraints for the autocorrelation function $f$ corresponding to $A$.

\medskip

Let $A$ be a Lebesgue measurable, {\em 1-avoiding set} in $\R^2$, that is, a measurable subset of the plane containing no two points at distance 1. Denote by $m_1(\R^2)$  the supremum of possible upper densities of such sets $A$ (for the rigorous definition, see Section~\ref{sec2}). Erd\H os conjectured in \cite{E82} that $m_1(\R^2)$ is less than $1/4$, a conjecture that has been open ever since.

One of the easiest upper bounds for $m_1(\R^2)$ is $1/3$, shown by the fact that $A$ may contain at most one of the vertices of any regular triangle of edge length~1. This simple idea was strengthened by Moser \cite{Mo61} using a special unit distance graph, the Moser spindle, implying that $m_1(\R^2) \leq 2/7 \approx 0.285$. Sz\'ekely \cite{Sze02} improved the upper bound to $\approx 0.279$. Applying Fourier analysis and linear programming Oliveira Filho and Vallentin \cite{OV} proved that $m_1(\R^2) \leq 0.268$, which was further improved to $\approx 0.259$ by Keleti, Matolcsi, Oliveira Filho and Ruzsa \cite{KMOR15}. Recently, Bellitto, P\^{e}cher and S\'{e}dillot \cite{BPS18} (see also Bellitto~\cite{B18}) used a purely combinatorial argument -- based on the fractional chromatic number of finite graphs -- to reach the currently best known bound of 0.25646, by constructing a large unit distance graph inspired by the work of de Grey \cite{G18+} on the chromatic number of the unit distance graph of $\R^2$. We revert here to the Fourier analytic method and prove the following improved bound, getting tantalizingly close to the conjecture of Erd\H os.

\begin{theorem}\label{thm}
Any Lebesgue measurable, 1-avoiding planar set has upper density at most $0.25442$.
\end{theorem}

Despite considerable efforts, these upper bounds are still very far from the largest lower bound for $m_1(\R^2)$, that is, $0.22936$, which is given by a construction of Croft \cite{Cr67}.

The question may be formulated in higher dimensions as well. The articles of Bachoc, A. Passuello, and A. Thiery \cite{BPT15} and of DeCorte, Oliveira Filho and Vallentin \cite{DOV} contain detailed historical accounts and a complete overview of recent results in that direction.

Perhaps the most famous related question is the Hadwiger-Nelson problem about the chromatic number $\chi(\R^2)$ of the plane: how many colours are needed to colour the points of the plane so that there is no monochromatic segment of length 1? Recently, de Grey \cite{G18+} proved that $\chi(\R^2) \geq 5$, a result which stirred up interest in this area.

\section{Subgraph constraints} \label{sec2}

Our proof is based on the techniques presented in  \cite{KMOR15} and \cite{DOV}, with an essential new ingredient of including triple-correlation constraints.

Let $A\subset \R^2$ be a measurable, 1-avoiding set. The {\em upper density of $A$}, denoted by $\overline{\delta(A)}$, is given by
\[
\overline{\delta(A)} = \limsup_{R \rightarrow \infty} \frac{\lambda_2(A \cap D(x, R))}{\lambda_2(D(x, R))},
\]
where $\lambda_2$ is the planar Lebesgue measure, and $D(x, R)$ denotes the disc of radius $R$ centered at $x$. The upper density is independent of the choice of $x\in \R^2$. In case the limit of the above quantity also exists, we call it the {\em density} of $A$, denoted by $\delta(A)$:
\[
\delta(A) = \lim_{R \rightarrow \infty} \frac{\lambda_2(A \cap D(x, R))}{\lambda_2(D(x, R))},
\]
which is again known to be independent of $x$. 

Our goal is to estimate
\[
m_1(\R^2) = \sup \{ \overline{\delta(A)} \, : \, A\subset \R^2 \textrm{ is 1- avoiding and measurable}  \}
\]
from above.

Due to a trivial argument taking limits \cite{KMOR15}, we may assume that $A$ is periodic with respect to a lattice $L \subset \R^2$, i.e. $A = A+L$. Measurable periodic sets always have densities. Moreover, $m_1(\R^2)$ may be approximated arbitrarily well by densities of 1-avoiding, measurable, periodic sets \cite{OV}. Therefore, we may restrict ourselves to this class when estimating $m_1(\R^2)$.

The {\em autocorrelation function} $f: \R^2 \rightarrow \R$ of $A$ is defined by
\begin{equation}\label{autof}
f(x) = \delta (A \cap (A-x)).
\end{equation}
Then $\delta(A) = f(0)$, and the fact that $A$ is 1-avoiding translates to the condition that $f(x)=0$ for all unit vectors $x$.

\medskip

To introduce some further notations, assume that $C$ is a finite set of points in the plane. $\binom{C}{i}$ will denote the set of $i$-tuples of distinct points of~$C$. Further, let
\begin{equation}\label{sigma}
\Sigma_i(C) = \sum_{\{x_1, \ldots, x_i \}\in \binom{C}{i}} \delta \left(  (A - x_1)\cap \ldots \cap (A - x_i) \right)
\end{equation}
and
\begin{equation}\label{sigmacirc}
\Sigma_i^\circ (C) = \sum_{\{x_1, \ldots, x_i \}\in \binom{C}{i}} \delta \left( A \cap (A - x_1)\cap \ldots \cap (A - x_i) \right)\,.
\end{equation}
By convention, $\Sigma_0 (C)=1$ and $\Sigma_0^\circ (C) = f(0)$. Note also that $\Sigma_1 (C)= |C| f(0)$, and $\Sigma_1^\circ (C)= \sum_{x \in C} f(x)$.
Obviously,
\begin{equation}\label{sigeq}
\Sigma_i^\circ (C)\leq \Sigma_i (C)
\end{equation} holds for every $i$.

\medskip

The estimate for $m_1(\R^2)$ of Keleti et al. \cite{KMOR15} relies on the following lemma. A graph is called a {\em unit distance graph} if its vertex set is a subset of $\R^2$, and its edges are given by the pairs of points being at distance~1. The independence number (i.e. the maximal number of independent vertices) of a graph $G$ is denoted by $\alpha(G)$. For simplicity, if not specified otherwise, we denote the vertex set of a graph $G$ by the same letter $G$, while the set of edges is denoted by $E(G)$.

\begin{lemma}[\cite{OV, Sze02} (cf. also \cite{KMOR15})]\label{lemma1}
Let $f$ be the autocorrelation function of a measurable, periodic, 1-avoiding set $A \subset \R^2$, as defined in \eqref{autof}. Then:
\begin{itemize}
\item[(C0)] $f(x) = 0$ for every $x \in \R^2$ with $|x|=1$;
\item[(C1)] If $G$ is a finite unit distance graph, then
\[
\sum_{x \in G} f(x) \leq \alpha(G)f(0);
\]
\item[(C2)] If $C \subset \R^2$ is a finite set of points, then
\[
\sum_{\{x,y\} \in \binom{C}{2}} f(x-y) \geq |C|f(0) -1.
\]
\end{itemize}
\end{lemma}

Constraint (C1) was first used by Oliveira and Vallentin \cite{OV}, while Sz\'ekely applied (C2) in \cite{Sze02}.

\medskip

We  will need a relaxed version of Lemma~\ref{lemma1}, which appeared in Section 7.1. of \cite{DOV} entitled as a {\em subgraph constraint}. As the actual formula is somewhat hard to extract from the discussion of \cite{DOV}, we include a short proof for convenience.

\begin{lemma}\label{lemma_subgraph}
Let $G$ be a finite graph with independence number $\alpha(G)$. Then
\begin{align}
\tag*{\hspace{10 pt}(C1R)} \sum_{x \in G} f(x) -  \sum_{\{ x, y \} \in E(G)} f(x - y) &\leq \alpha(G)f(0).
\end{align}
\end{lemma}

Note that
we may recover  condition (C1) of Lemma~\ref{lemma1} by setting $G$ to be a unit distance graph in (C1R).

\begin{proof}
Consider the translated sets $A-x$ for every $x \in G$. For any point $z \in A$ consider the function
\[
g(z)=|\{x\in G: z\in (A-x)\}|-|\{\{x,y\}\in E(G): z\in (A-x)\cap (A-y)\}|.
\]
Loosely speaking, $g(z)$ counts the number of times $z$ is being covered by translates of $A$ corresponding to vertices of $G$, minus the number of times it is covered by translates corresponding to edges of $G$. We claim that for each $z\in A$,  $g(z)\le \alpha(G)$ holds. To see this, let $v=|\{x\in G: z\in (A-x)\}|$ and $e=|\{\{x,y\}\in E(G): z\in (A-x)\cap (A-y)\}|$, so that $g(z)=v-e$. The vertices $\{x\in G: z\in (A-x)\}$ span a subgraph $G'$ of $G$. Let $G_1, \dots, G_c$ denote the connected components of $G'$. Clearly, the number of components satisfies $c\le \alpha(G)$. Let $v_i$ and $e_i$ denote the number of vertices and edges in $G_i$, respectively. We always have $e_i\ge v_i -1$ (with equality holding if and only if $G_i$ is a tree). Therefore, $e=e_1+\dots +e_c\ge (v_1-1)+\dots +(v_c-1)=v-c\ge v-\alpha(G)$, which proves $g(z)\le \alpha(G)$.

\medskip

Integrating the inequality $g(z)\le \alpha(G)$ over $A$ (with an obvious limiting process, as $A$ is unbounded), we obtain
\[
\sum_{x\in G} \delta(A\cap(A-x)) - \sum_{\{ x, y \} \in E(G)} \delta(A\cap (A-x)\cap (A-y))\le \alpha(G)f(0).
 \]
Finally, noting that $\delta(A\cap(A-x))=f(x)$ and $\delta(A\cap (A-x)\cap (A-y))\le \delta((A-x)\cap (A-y)) =f(x-y)$, we obtain (C1R).
\end{proof}

\section{Triple correlations}\label{sec_triple}

We continue with estimates involving higher order correlations between the points of $A$. The proof of~(C2), as in  \cite{Sze02, KMOR15}, is based on the inclusion-exclusion principle:
\begin{align*}
1 &\geq \delta \left( \bigcup_{x \in C} (A - x) \right)\\
&\geq \sum_{x \in C} \delta(A - x) - \sum_{\{x,y\} \in \binom{C}{2}} \delta((A-x) \cap (A - y))\\
&= |C| \delta(A) - \sum_{\{x,y\} \in \binom{C}{2}} \delta(A \cap (A-(x-y)))\,.
\end{align*}
Note that at the second inequality above, intersections of three or more sets are omitted. We will make use of the natural idea to take into account triple intersections, which is equivalent to studying the density of prescribed triangles in $A$. We will then use these estimates to obtain new linear constraints on the autocorrelation function $f$.

\medskip

First, we set an upper bound for triangle densities.

\begin{lemma}\label{lemma_triangleupper}
Assume that $G\subset \R^2$ is a finite unit distance graph with $\alpha(G) \le 3$. Then

\begin{align}
\tag*{\hspace{10 pt}(T1)} \Sigma_3(G) \leq 1 - |G| f(0) + \sum_{\{x,y\} \in \binom{G}{2} } f(x - y)\,.
\end{align}
\end{lemma}

\begin{proof}
Since $\alpha(G) \le 3$, $\Sigma_i(G) = 0$ holds for every $i \geq 4$. Thus, by the inclusion-exclusion principle,
\begin{align*}
1 \geq \delta \left(\bigcup_{x \in G}( A - x)\right) &= \Sigma_1(G) - \Sigma_2(G) + \Sigma_3(G) \\
 &=|G| f(0) -  \sum_{\{x,y\} \in \binom{G}{2}} f(x - y) + \Sigma_3(G)\, .
 \qedhere
\end{align*}
\end{proof}

Next, we derive a lower bound for triangle densities.

\begin{lemma}\label{lemma_trianglelower2}
If $G$ is a finite unit distance graph with $\alpha(G) \le 3$, then

\begin{align}
\tag*{\hspace{10 pt}(T2)} \Sigma_3(G) \geq \Sigma_3^\circ(G) \geq \sum_{x \in G}f(x) - 2 f(0).
\end{align}
\end{lemma}

\begin{proof}
The first inequality is trivial, as noted in \eqref{sigeq}. To see the second, consider the sets $G_x=A\cap (A-x)$, and take an arbitrary point $z\in A$.  The point $z$ can be contained in at most three $G_x$'s, because $\alpha(G)\le 3$. The total density of points covered by three $G_x$'s is exactly $\Sigma_3^\circ (G)$. All other points in $A$ are covered by at most two $G_x$'s. Also, the density of $G_x$ is $f(x)$, and $\delta(A)=f(0)$, by definition. Therefore, $\sum_{x \in G}f(x)\le 2 f(0)+\Sigma_3^\circ(G)$.
\end{proof}


We now turn to defining the geometric configurations to which inequalities (T1) and (T2) will be applied. Note that while the statements of Lemma \ref{lemma_triangleupper} and Lemma \ref{lemma_trianglelower2} are fairly trivial, it is not straightforward to find some geometric configurations such that conditions (T1) and (T2) yield non-trivial new constraints on the autocorrelation function $f(x)$. The search for such configurations is almost like looking for a needle in a haystack, and we cannot point out any general method to succeed.

\medskip

Let $\theta\in [0,2\pi]$ and consider the following eight points in the plane (see Figure~\ref{fig1}): $V_1=(0, 0)$, $V_2=(\frac{\sqrt{3}}{2}, \frac{1}{2})$, $V_3=(\frac{\sqrt{3}}{2}, -\frac{1}{2})$, $V_4=(\sqrt{3},0)$,  $V_5=(\cos \theta, \sin \theta)$, $V_6=(\frac{\sqrt{3}}{2} + \cos \theta, \frac{1}{2} + \sin \theta)$, $V_7=(\frac{\sqrt{3}}{2} + \cos \theta, -\frac{1}{2} + \sin \theta)$, $V_8=(\sqrt{3}+\cos \theta, \sin \theta)$.

\begin{figure}[h]
  \centering
  \includegraphics[width = 0.95 \textwidth]{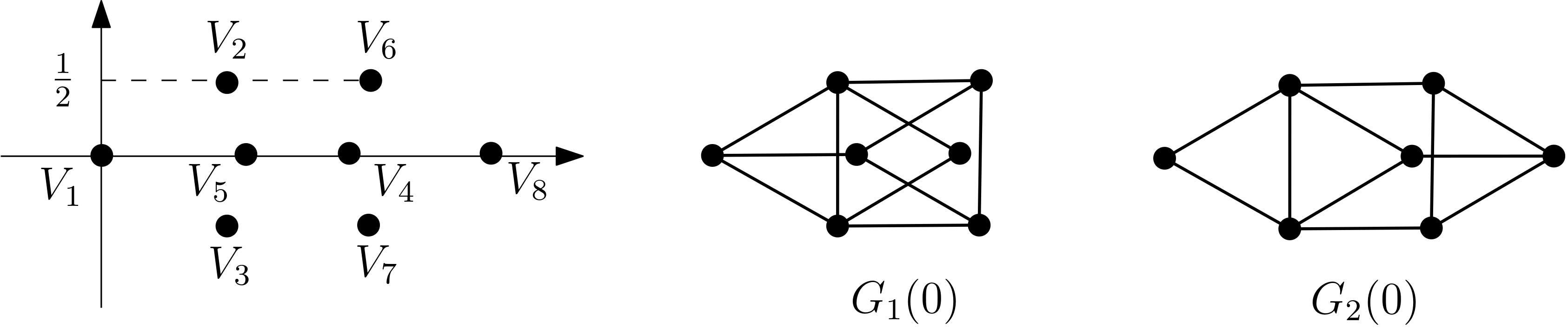}
  \caption{The unit distance graphs $G_1(\theta)$ and $G_2(\theta)$ for $\theta = 0$.  }
  \label{fig1}
\end{figure}

Consider the two unit distance graphs with vertex sets
\begin{equation}\label{g1def}
  G_1=G_1(\theta)=\{V_1, V_2, V_3, V_4, V_5, V_6, V_7\}
\end{equation}
and
\begin{equation}\label{g2def}
  G_2=G_2(\theta)=\{V_1, V_2, V_3, V_4, V_6, V_7, V_8\},
\end{equation}
and all pairs of vertices at distance 1 being connected with an edge.
Notice that for all values of $\theta$, both $G_1$ and $G_2$ have independence number $\alpha=3$, and both of them contain the same two independent triangles:  $(V_1, V_4, V_6)$ and $(V_1, V_4, V_7)$. Therefore, $\Sigma_3(G_1)=\Sigma_3(G_2)$, which we commonly denote by~$\Sigma_3$.

\medskip

We apply Lemma \ref{lemma_triangleupper} to $G_1$ to obtain
\[
\Sigma_3 \leq 1 - 7f(0) + \sum_{\{x,y\} \in \binom{G_1}{2} } f(x - y),
 \]
while Lemma \ref{lemma_trianglelower2} applied to $G_2$ implies that
\[
\Sigma_3\ge \sum_{x\in G_2} f(x)- 2f(0).
\]
Comparing these two estimates leads to
\begin{align}
\tag*{\hspace{10 pt}(CT)} \sum_{x\in G_2} f(x)\le 1 - 5f(0) + \sum_{\{x,y\} \in \binom{G_1}{2} } f(x - y).
\end{align}
This constraint turns out to be surprisingly powerful.

\medskip

It is natural to wonder whether sharper bounds on $m_1(\R^2)$ could be reached by imposing further conditions on $f$, possibly coming from $4$-tuple, $5$-tuple, etc., correlations of the set $A$. The answer is provided by Theorem 1.1 and Theorem 7.3 in \cite{DOV}, which state that if we write up all {\it complete positivity constraints} or all {\it Boolean quadratic constraints} on the function $f$, then the implied upper bound on the density of $A$ will converge to $m_1(\R^2)$. This means, in theory, that {\it this method is guaranteed to succeed} in proving the conjecture $m_1(\R^2)<0.25$, if the inequality is true. In practice, however, the Boolean quadratic cone has so many facets even in relatively small dimensions that it is hopeless to add them all in any kind of numerical computation. For this reason, one is restricted to finding "clever" new constraints by geometric intuition, such as (CT) above. In comparison, we are not aware of such a theoretical guarantee of success for the method of fractional chromatic numbers of \cite{BPS18}: as far as we know, it may well happen that the fractional chromatic number of any finite unit distance graph is smaller than 4, while $m_1(\R^2)<0.25$.

\section{Fourier analysis and linear programming} \label{sec_linprog}

The detailed description of the Fourier analytic method can be found in \cite{KMOR15}, we will only summarize the essentials here. We remind the reader that the 1-avoiding set $A$ is assumed to be periodic with a period lattice $L$. This enables us to perform a Fourier expansion of $f(x)=\delta(A\cap (A-x))$ in the Hilbert space $L^2(\R^2/L)$.

\medskip

We also apply a standard trick of averaging. Note that all the inequalities stated in constraints (C1), (C2), (C1R) and (CT) hold for all rotated copies of a given graph. Thus, they may be averaged over the orthogonal group $O(2)$ of the plane. We will use the notation $\mathring{f}(x)$ for the radial average of $f(x)$:
\begin{equation}\label{f_av}
  \mathring{f}(x) = \frac 1 {2 \pi} \int_{S^1} f(\xi |x|) d \omega(\xi),
\end{equation}

where $\omega$ is the perimeter measure on the unit circle $S^1$. The advantage of this averaging is that $\mathring{f}$ is radial, i.e. $\mathring{f}(x)$ depends only on $|x|$. Also, the above remark shows that the constraints (C1), (C2), (C1R) and (CT) remain valid for the function $\mathring{f}$.

\medskip
As usual, the Bessel function of the first kind with parameter 0, $\O_2(|x|)$,  is defined as
\[
\O_2(|x|) = \frac 1 {2 \pi} \int_{S^1} e^{ i x \xi} d \omega(\xi),
\]
As explained in \cite{KMOR15},
\[
\mathring{f}(x) =\sum_{u \in 2 \pi L^*}\widehat{f}(u) \Omega_2( |u| |x|),
\]
where $L^*$ denotes the dual lattice of $L$. Introducing the notation
\[
\kappa(t) = \sum_{u \in 2 \pi L^* , |u| = t} \widehat{f}(u),
\]
the previous equation simplifies to
\begin{equation}\label{fkappa}
\mathring{f}(x) = \sum_{t \geq 0} \kappa(t) \Omega_2( t |x|),
\end{equation}
where the summation is taken for those values of $t$ which come up as a length of a vector in $2\pi L^*$.

\medskip

Introduce the notations $\delta = \delta(A)$ and $\widetilde{\kappa}(t) = \frac {\kappa(t)}{\delta}$. Conditions $f(x) \geq 0$, $f(0) = \delta$, (C0), (C1R) and (CT) via \eqref{f_av} and \eqref{fkappa} lead to the following properties of the function $\wk(t)$ (see \cite{KMOR15} for details):

\begin{itemize}\label{wk}
\item[($\widehat{\textrm{CP}}$)] $\wk(t) \geq 0$  for every $t \geq 0$,
\item[($\widehat{\textrm{CS}}$)]$ \sum_{t \geq 0} \wk(t) =1$,
  \item [($\widehat{\textrm{C0}}$)] $ \sum_{t \geq 0} \wk(t) \Omega_2( t ) =0 $,
    \item [($\widehat{\textrm{C1R}}$)] For every finite graph $G$,
  \[
  \sum_{t \geq 0} \wk(t) \left( \sum_{x \in G} \Omega_2( t |x|) -  \sum_{\{x,y\} \in E(G)}  \Omega_2( t |x -y|)\right) \leq \alpha(G)
  \]
  \item [($\widehat{\textrm{CT}}$)] For any $\theta\in [0,2\pi]$, and the graphs $G_1(\theta)$ and $G_2(\theta)$ defined in Section~\ref{sec_triple} by \eqref{g1def} and \eqref{g2def},
  \[
  \sum_{t \geq 0} \wk(t) \left( \sum_{\{x,y\} \in \binom{G_1}{2}} \Omega_2( t |x -y|)-\sum_{x\in G_2} \Omega_2( t |x|) \right) \geq 5 -\frac 1 {\delta}\,.
  \]
\end{itemize}

\noindent

Forget, for a moment, that $\delta=\delta(A)$, and just fix any particular value of $\delta>0$.  Consider the coefficients $\wk(t)$  (for $t\ge 0$) as variables in the continuous linear program
\begin{align} \label{LP}
\begin{split}
&\textrm{maximize } \wk(0)\\
&\textrm{subject to } (\widehat{\textrm{CP}}), (\widehat{\textrm{CS}}), (\widehat{\textrm{C0}}), (\widehat{\textrm{C1R}}),(\widehat{\textrm{CT}}).
\end{split}
\end{align}

Let $s=\sup \wk(0)$ denote the solution of this LP-problem. If, for a given value of $\delta$, there exists a 1-avoiding set $A$ with density $\delta$, then there exists a system of values $\wk(t)$ satisfying \eqref{LP} such that $\wk(0)=\delta$. Therefore, in such a case, $s\ge \delta$. Conversely, if for a given value of $\delta$ we find that $s<\delta$, then we may conclude that no 1-avoiding set with density $\delta$ exists, therefore, $m_1(\R^2)\le \delta$. By linear programming duality, the inequality $s<\delta$ may be testified by the existence of a witness function.

\begin{proposition}\label{lp-bound}
Let $\Gcal$ be a finite family of finite graphs in~$\R^2$, $\Tcal$ be a finite collection of angles in $[0, 2\pi]$, and for each $\theta\in \Tcal$ consider the unit distance graphs $G_1(\theta), G_2(\theta)$ defined in Section \ref{sec_triple} by \eqref{g1def} and \eqref{g2def}. Suppose that for some non-negative numbers ~$v_0$,
$v_1$,  $w_G$ for $G \in \Gcal$ and $w_\theta$ for $\theta \in \Tcal$
the function $W(t)$ defined by
\begin{align}
\label{witness}
\begin{split}
W(t) &= v_0 + v_1 \Omega_2(t) \\
& +  \sum_{G \in \Gcal} w_G \left( \sum_{x \in G} \Omega_2( t |x|) - \sum_{\{x,y\} \in E(G) }  \Omega_2( t |x -y|) \right) \\
&- \sum_{\theta \in \Tcal} w_{\theta} \left( \sum_{\{x,y\} \in \binom{G_1(\theta)}{2}} \Omega_2(t|x-y|)- \sum_{x \in G_2(\theta)} \Omega_2( t |x|) \right)\\
\end{split}
\end{align}
satisfies $W(0) \geq 1$ and $W(t) \geq0$ for $t>0$.

Then $m_1(\R^2) \leq \delta$, where $\delta$ is the positive solution of the equation
\begin{equation}\label{deltaup}
\delta^2 = \delta \Big( v_0 + \sum_{G \in \Gcal} w_G \alpha(G) - 5\sum_{\theta \in \Tcal} w_\theta    \Big) +   \sum_{\theta \in \Tcal} w_\theta .
\end{equation}
\end{proposition}

\begin{proof}

For any function $W(t)$ satisfying
$W(0) \geq 1$ and $W(t) \geq0$ for $t>0$ we have
\begin{equation}\label{alpha0}
 \delta = \wk(0) \leq \sum_{t \geq 0} \wk(t) W(t).
\end{equation}
If $W(t)$ is in the form \eqref{witness}, then inequalities $(\widehat{\textrm{CP}}), (\widehat{\textrm{CS}}),  (\widehat{\textrm{C0}}),(\widehat{\textrm{C1R}})$ and \eqref{alpha0} imply
\begin{equation}\label{deltaup1}
\delta \leq   v_0 + \sum_{G \in \Gcal} w_G \alpha(G) - 5\sum_{\theta \in \Tcal} w_\theta + \frac{1}{\delta}\sum_{\theta \in \Tcal} w_\theta .
\qedhere
\end{equation}
\end{proof}

\section{Numerical bounds}\label{sec_num}

\noindent
As indicated in Proposition \ref{lp-bound} above,  we will use two types of constraints, (C1R) and (CT), in addition to the trivial ones. Constraint (C1R) will be applied to certain isosceles triangles in the plane. Constraint (CT) will be applied, with particular choices of the angle $\theta$, to the graphs $G_1(\theta), G_2(\theta)$ defined by \eqref{g1def} and \eqref{g2def}.

\medskip

In order to handle the linear program numerically, we use a discrete approximation. Based on the previous results, we only search for the coefficients $\wk (t_i)$, where $t_i = i \eps_0$, with $\eps_0 = 0.05$ and $i \leq 12000$, thus, $t_i \in [ 0, 600 ]$. For all other values of $t \geq 0$, we set $\wk(t) = 0$. The error resulting from the discretization is corrected in the last step of the algorithm.

Finding suitable triangles and graphs which yield strong upper bounds on $m_1(\R^2)$ is a tedious task, where we utilized a bootstrap algorithm. Once a given set of constraints is fixed, and the corresponding linear program is solved, one has to numerically search for configurations of points for which (C1R) or (CT) is violated. Adding these to the list of constraints, and dropping the non-binding ones, the same procedure may be repeated until no significant improvement may be obtained. In its polished form, our construction uses 15 nontrivial linear constraints: 10 of the type (C1R) and 5 of type (CT).

The family $\Gcal$ used for the estimate consists of 10 triangles of the form $\{(x_1, 0), (x_2, y), (x_2, -y)\}$, with the triples $(x_1, x_2, y)$ being listed in Table~\ref{tab:triangle}. Constraint (CT) is applied to the graphs defined by \eqref{g1def} and \eqref{g2def} with the values of $\theta$ ranging over the family of $\Tcal$, which is listed in Table~\ref{tab:theta}. In order to avoid errors stemming from numerical computations, all the non-zero norms and distances between points of the configurations are chosen to be at least 0.1.

\begin{figure}[h!]
  \centering
  \includegraphics[width = 0.5 \textwidth]{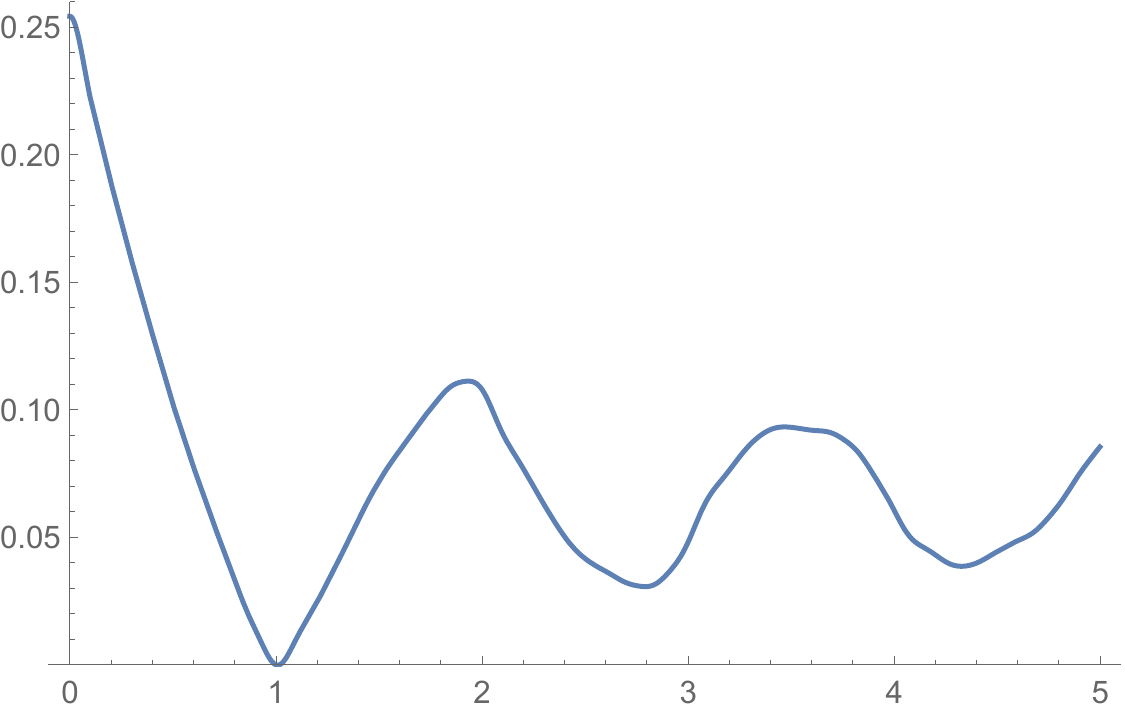}
  \caption{The  function $\mathring{f}(x) / \delta$. }
  \label{fig_f0}
\end{figure}

Using these graphs, we construct the witness function $W(t)$ as in \eqref{witness} 
with the coefficients described in Table~\ref{tab:coeff}.  It is easy to check numerically that $W(t)$ satisfies the required properties. Technical details about the rigorous verification of this are described in~\cite{KMOR15}.

With this construction of $W(t)$, the quadratic equation \eqref{deltaup} takes the form
\[
\delta^2  + 7.188702 \, \delta - 1.893645= 0,
\]
whose positive solution is $\delta = 0.254416$.

The coefficients $\wk(t)$ obtained as the solution of the linear program \eqref{LP} also provide the normalized, radialized autocorrelation function $\mathring{f}(x)/\delta$ via equation \eqref{fkappa}. This function  could, in principle, be the autocorrelation of a hypothetical 1-avoiding set $A$ with density $\delta=0.254416$. The function is plotted in Figure~\ref{fig_f0}.

\section{Acknowledgement}

The authors are grateful to F. M. Oliveira Filho and Th. Bellitto for the inspiring conversations, and for the anonymous referee for providing helpful suggestions.

\section{Appendix: Numerical values}

\begin{table}[h!]
\bgroup
\setbox0=\hbox{$(0)$}
\newdimen\h \h=\ht0
\newdimen\d \d=\dp0

\advance\h by 6pt
\advance\d by 6pt

\def\hr{\vrule height\h width0pt depth0pt}
\def\dr{\vrule height0pt width0pt depth\d}
\hbox to\hsize{\hss
\begin{tabular}{ll|ll}
\hline
\hr \dr
$G_1$&$\{-0.123996, 1.946331, 0.501521\}$&$G_2$&$\{-0.157711, 0.542869, 0.499760\}$\\
\hline
\hr \dr
$G_3$&$\{0.553873, -0.276937, 0.479669\}$&$G_4$&$\{-0.424898, 0.382590, 0.490199\}$\\
\hline
\hr \dr
$G_5$&$\{2.70637, 1.842120, 0.506318\}$&$G_6$&$\{-0.955984, 0.026128, 0.112481\}$\\
\hline
\hr \dr
$G_7$&$\{-0.767499, 0.143459, 0.340280\}$&$G_8$&$\{0.476394, -0.337821, 0.486967\}$\\
\hline
\hr \dr
$G_9$&$\{0.668340, -0.199610, 0.428893\}$&$G_{10}$&$\{-0.177622, 0.519323, 0.499597\}$\\[2 pt]
\hline
\end{tabular}
\hss}
\egroup
\bigskip

\caption{Triples $\{x_1, x_2, y \}$ corresponding to the family~$\Gcal$.
   }
\label{tab:triangle}
\end{table}

\vspace{-10 pt}
\begin{table}[h!]
\bgroup
\setbox0=\hbox{$(0)$}
\newdimen\h \h=\ht0
\newdimen\d \d=\dp0

\advance\h by 6pt
\advance\d by 6pt

\def\hr{\vrule height\h width0pt depth0pt}
\def\dr{\vrule height0pt width0pt depth\d}
\hbox to\hsize{\hss
\begin{tabular}{ll|ll|ll|ll|ll}
\hline
\hr \dr $\theta_1$& $1.851176$  &$\theta_2$   & 1.864223  & $\theta_3$ & 1.911210 & $\theta_4$ & 1.935475 & $\theta_5$ & 1.954980 \\
\hline
\end{tabular}
\hss}
\egroup
\bigskip

\caption{Angles in the family $\Tcal$. }
\label{tab:theta}
\end{table}

\vspace{-10 pt}

\begin{table}[h!]
\bgroup
\setbox0=\hbox{$(0)$}
\newdimen\h \h=\ht0
\newdimen\d \d=\dp0

\advance\h by 6pt
\advance\d by 6pt

\def\hr{\vrule height\h width0pt depth0pt}
\def\dr{\vrule height0pt width0pt depth\d}
\hbox to\hsize{\hss
\begin{tabular}{ll|ll|ll}
\hline
\hr \dr $v_0$& $1.4024971970$  &$v_1$   & \hspace{-2 mm}10.9609841893  & $w_{G_1}$ & 0.1938457698 \\
\hline
\hr \dr $w_{G_2}$& $ 0.2751221022$  &$w_{G_3}$  & 0.5079791712  & $w_{G_4}$ & 0.3069034307 \\
\hline
\hr \dr $w_{G_5}$& $ 0.3404898985$  &$w_{G_6}$  & 0.3361763782  & $w_{G_7}$ & 0.1961680281 \\
\hline
\hr \dr $w_{G_8}$& $ 0.0133266364$  &$w_{G_9}$  & 0.5532445066  & $w_{G_{10}}$ & 0.0474157478 \\
\hline
\hr \dr $w_{\theta_1}$& $ 0.3055968204$  &$w_{\theta_2}$  & 0.6557537159  & $w_{\theta_3}$ & 0.1173616739 \\
\hline
\hr \dr $w_{\theta_4}$& $ 0.5306336291$  &$w_{\theta_5}$  & 0.2842993917 & & \\
\hline
\end{tabular}
\hss}
\egroup
\bigskip

\caption{Coefficients of the witness function $W(t)$. }
\label{tab:coeff}
\end{table}

\end{document}